\documentclass[11pt]{amsart}
\usepackage{amsmath,amsthm,amssymb, amscd, amsfonts}
\usepackage[all]{xy}
\usepackage{leftidx}
\usepackage[latin1]{inputenc}
\usepackage{enumerate}
\usepackage[dvipsnames]{xcolor}

\theoremstyle{plain}

\newtheorem{theorem}{Theorem}[section]
\newtheorem{corollary}[theorem]{Corollary}

\newtheorem{proposition}[theorem]{Proposition}
\newtheorem{lemma}[theorem]{Lemma}
\newtheorem{question}[theorem]{Question}
\theoremstyle{definition}

\newtheorem{example}[theorem]{Example}

\theoremstyle{remark}
\newtheorem{remark}[theorem]{Remark}

\numberwithin{equation}{section}\theoremstyle{plain}

\newcommand{\I}{\mathcal{I}}

\renewcommand{\1}{\textbf{1}}

\newcommand{\A}{{\mathcal A}}
\newcommand{\B}{{\mathcal B}}
\newcommand{\C}{{\mathcal C}}
\newcommand{\D}{{\mathcal D}}

\newcommand{\Y}{{\mathcal Y}}

\newcommand{\Z}{{\mathcal Z}}

\newcommand{\E}{{\mathcal E}}
\newcommand{\U}{{\mathcal U}}

\newcommand{\Rep}{\operatorname{Rep}}
\newcommand{\cd}{\mathrm{cd}}

\newcommand\Irr{\operatorname{Irr}}
\newcommand\FPdim{\operatorname{FPdim}}

\newcommand\vect{\operatorname{Vec}}
\newcommand\svect{\operatorname{sVec}}
\newcommand\id{\operatorname{id}}

\newcommand\Hom{\operatorname{Hom}}
\newcommand\rk{\operatorname{rk}}

\begin{document}
\title{Braided extensions of a rank $2$ fusion category}
\author{Jingcheng Dong}
\email{jcdong@nuist.edu.cn}
\address{College of Mathematics and Statistics, Nanjing University of Information Science and Technology, Nanjing 210044, China}

\author{Hua Sun}
\address{Department of Mathematics, Yangzhou University, Yangzhou, Jiangsu 225002, China}
\email{d160028@yzu.edu.cn}




\date{\today}

\begin{abstract}
We classify braided extensions $\C$ of a rank $2$ fusion category. The result shows that $\C$ is tensor equivalent to a Deligne's tensor product of some known categories, except $\C$ is slightly degenerate and generated by a $\sqrt{2}$-dimensional simple object. To start with, we describe the fusion rules, universal grading group, and the Frobenius-Perron dimensions of simple objects of $\C$ without the restriction that $\C$ is braided.
\end{abstract}

\maketitle

\section{Introduction}\label{sec1}
Fusion categories arise from many areas of mathematics and physics, including representation theory of semisimple Hopf algebras and quantum groups \cite{BaKi2001lecture}, vertex operator algebras \cite{2016DongWang} and topological quantum field theory \cite{Turaer1994}. By definition, a fusion category $\C$ is a $k$-linear semisimple rigid tensor category with finitely many isomorphism classes of simple objects, finite-dimensional vector spaces of morphisms and the unit object $\1$ is simple.

Group extensions are important notion in the theory of fusion categories. Let $G$ be a finite group. A fusion category $\C$ is $G$-graded if there is a decomposition $\C=\oplus_{g\in G}\C_g$ of $\C$ into a direct sum of full abelian subcategories, such that $(\C_g)^{*}=\C_{g^{-1}}$ and $\C_g\otimes \C_h\subseteq \C_{gh}$ for all $g,h \in G$. The grading $\C=\oplus_{g\in G}\C_g$ is called faithful if $\C_g\neq 0$ for all $g\in G$. A fusion category $\C$ is called a $G$-extension of $\D$ if there is faithful grading $\C=\oplus_{g\in G}\C_g$, such that the trivial component $\C_e$ is equivalent to $\D$, where $e$ is the identity of $G$.

Group extensions play an important role in the study and classification of fusion categories, see \cite{2016Dongbraided,etingof2010fusion,gelaki2008nilpotent} for example. These studies lead to a natural question:
\begin{question}
How much can be said about an extension of a given fusion category by any group?
\end{question}

The answer in the extreme case of $\D=\vect$ is that the extension $\C$ must be a pointed fusion category. For most examples, we do not know the answer to this question. In this paper, we make an attempt to study the first level of complexity; that is, braided extensions of a rank $2$ fusion category $\D$.

Recall from \cite{ostrik2003fusion} that a rank $2$ fusion category is either pointed or a Yang-Lee category. Let $\C$ be a $G$-extension of a rank $2$ fusion category. The main work of this paper is to describe the structure of $\C$ without any restriction on $G$, and apply this description to the classification of $\C$ when it is braided. Our main result is stated as follows.

\begin{theorem}
Let $\C$ be an extension of a rank $2$ fusion category. Assume that $\C$ is braided. Then $\C$ is exactly one of the following:

(1)\,  $\C$ is pointed.

(2)\,  $\C\cong \I\boxtimes \B$, where $\I$ is an Ising category, $\B$ is a braided pointed fusion category.

(3)\, $\C\cong \C_{ad}\boxtimes \C_{pt}$, where the adjoint subcategory $\C_{ad}$ is a Yang-Lee category, $\C_{pt}$ is the largest pointed fusion subcategory of $\C$.

(4)\, $\C$ is slightly degenerate and generated by a $\sqrt{2}$-dimensional simple object. In this case, $\C$ is prime.
\end{theorem}

The organization of the paper is as follows. In section $2$ we recall some  basic results, notions and examples of fusion categories. In particular, we get the relationship between the order of $G(\C)$ and $\U(\C)$ when $\C$ is non-degenerate or slightly degenerate, where $G(\C)$ is the  group generated by isomorphism classes of invertible objects of $\C$, $\U(\C)$ is the universal grading group of $\C$. We prove that if $\C$ is non-degenerate then $|G(\C)|=|\U(\C)|$; if $\C$ is slightly degenerate then $|G(\C)|=|\U(\C)|$ or $2|U(\C)|$.

In Section \ref{sec3}, we study the $G$-extensions $\C$ of a rank $2$ pointed fusion category $\vect_{\mathbb{Z}_2}^{\omega}$. We prove that if $\C$ is not pointed then $\omega$ is trivial and $\C$ is a generalized Tambara-Yamagami fusion category. In this case, we obtain that the Frobenius-Perron dimension of every simple object of $\C$ is $1$ or $\sqrt{2}$ and the trivial component $\vect_{\mathbb{Z}_2}^{\omega}$ is exactly the adjoint subcategory $\C_{ad}$ of $\C$. We also obtain a sufficient and necessary condition for $\C$ containing an Ising category.

In Section \ref{sec4}, we study the $G$-extensions $\C$ of a Yang-Lee category. We prove that the Frobenius-Perron dimension of every simple object of $\C$ is $1$ or $\frac{1+\sqrt{5}}{2}$. We also obtain the fusion rules of $\C$. In particular, the Yang-Lee category is exactly the adjoint subcategory $\C_{ad}$ of $\C$ and the universal grading group $\U(\C)$ is isomorphic to the group $G(\C)$ of isomorphism classes of invertible objects of $\C$. Moreover, we prove that there is a 1-1 correspondence between the non-pointed fusion subcategories of $\C$ and the subgroups of the universal grading group $\mathcal{U}(\C)$.

In Section \ref{sec5}, we are devoted to find a non-degenerate fusion subcategory so that M\"{u}ger's decomposition theorem (see Theorem \ref{MugerThm}) can be used. It is easy when $\C$ is a braided extension of a Yang-Lee category $\Y$, since $\Y$ is non-degenerate. But the situation becomes difficult when $\C$ is a braided extension of $\vect_{\mathbb{Z}_2}^{\omega}$. We prove, except $\C$ is slightly degenerate and generated by a $\sqrt{2}$-dimensional simple object, $\C$ contains a Ising category which is non-degenerate. When $\C$ is slightly degenerate and generated by a $\sqrt{2}$-dimensional simple object, it is prime; that is, it can not contain any non-degenerate fusion subcategory.

\section{Preliminaries}\label{sec2}
Throughout this paper we shall work over an algebraically closed field $k$ of characteristic zero. Let $\C$ be a fusion category. The set of isomorphism classes of simple objects in $\C$ will be denoted by $\Irr(\C)$. It is a basis of the Grothendieck ring of $\C$. The cardinal number of $\Irr(\C)$ is called the rank of $\C$ and is denoted by $\rk(\C)$. The group of isomorphism classes of invertible objects of $\C$ will be denoted by $G(\C)$.

Let $\C$ be a fusion category. The Frobenius-Perron dimension $\FPdim(X)$ of a simple object $X \in \C$ is defined as the Frobenius-Perron eigenvalue of
the matrix of left multiplication by the class of $X$ in the basis $\Irr(\C)$ of the Grothendieck ring of $\C$. It is known that $\FPdim(X)\geq 1$ for all nonzero $X\in \C$, see \cite[Remark 8.4]{etingof2005fusion}. The Frobenius-Perron dimension of $\C$ is $\FPdim (\C)=\sum_{X \in \Irr(\C)} (\FPdim X)^2$.

\subsection{Pointed fusion category}\label{sec2.0}
A simple object $X\in \C$ is called invertible if $X\otimes X^*= \1$, where $X^*$ is the dual of $X$. This implies that $X$ is invertible if and only if $\FPdim(X)=1$. A fusion category is called pointed if all simple objects are invertible. Any pointed fusion category is equivalent to the category of finite dimensional vector spaces graded by some finite group $G$ with associativity determined by a cohomology class $\omega\in H^3(G,k^*)$, which is denoted by $\vect_{G}^{\omega}$. The largest pointed fusion subcategory of a fusion category $\C$ will be denoted by $\C_{pt}$.

\subsection{Universal grading}\label{sec2.1}
Let $\C$ be a fusion category. The adjoint subcategory $\C_{ad}$ is the full tensor subcategory of $\C$ generated by simple objects in $X\otimes X^{*}$ for all $X\in\Irr(\C)$. It is known that every fusion category has a canonical faithful grading $\C=\oplus_{g\in \U(\C)}\C_g$ with trivial component $\C_{ad}$ \cite{gelaki2008nilpotent}. This grading is called the universal grading of $\C$ and $\U(\C)$ is called the universal grading group of $\C$.

Any faithful grading of $\C$ by a finite group $G$ comes from a surjective group homomorphisms $\pi:\U(\C)\rightarrow G$ \cite[Corollary 3.7]{gelaki2008nilpotent}. Hence $\C_{ad}$ is a fusion subcategory of $\C_e$ for any faithful grading $\C=\oplus_{g\in G}\C_g$, where $e$ is the identity of $G$.

Let $\C=\oplus_{g\in G}\C_g$ be a $G$-extension of $\D$. Then the Frobenius-Perron dimensions of $\C_g$ are all equal and we have $\FPdim(\C)=|G|\FPdim(\D)$ \cite[Proposition 8.20]{etingof2005fusion}.

\subsection{Braided fusion categories}\label{sec2.2}
A braided fusion category $\C$ is a fusion category admitting a braiding $c$, where the braiding is a family of natural isomorphisms: $c_{X,Y}$:$X\otimes Y\rightarrow Y\otimes X$ satisfying the hexagon axioms for all $X,Y\in\C$.

\medbreak
Let $\C$ be a braided fusion category. Then $X,Y\in \C$ are said to centralize each other if $c_{Y,X}c_{X,Y}=\id_{X\otimes Y}$. The centralizer $\D'$ of a fusion subcategory $\D\subseteq \C$ is the full subcategory of objects which centralize every object of $\D$, that is
$$\D^{\prime}=\{X\in \C \mid c_{Y,X}c_{X,Y}=\id_{X\otimes Y},\forall\ Y\in \D\}.$$

The M\"{u}ger center $\Z_2(\C)$ of a braided fusion category $\C$ is the centralizer $\C'$ of $\C$ itself. A braided fusion category $\C$ is called non-degenerate if $\Z_2(\C)$ is equivalent to the trivial category $\vect$. A braided fusion category $\C$ is called slightly degenerate if $\Z_2(\C)$ is equivalent to the category $\svect$ of super-vector spaces.

\medbreak
A braided fusion category $\C$ is called symmetric if $\Z_2(\C)=\C$. Hence the M\"{u}ger center of a braided fusion category is a symmetric fusion category.

A symmetric fusion category $\C$ is called Tannakian if it is equivalent to the category $\Rep(G)$ of finite-dimensional representations of a finite group $G$, as braided fusion categories.

Let $\C$ be a symmetric fusion category. Deligne proved that there exist a finite group $G$ and a central element $u$ of order $2$, such that $\C$ is equivalent to the category $\Rep(G,u)$ of representations of $G$ on finite-dimensional super vector spaces, where $u$ acts as the parity operator \cite{deligne1990categories}. The following result is due to Drinfeld et al.

\begin{theorem}{\cite[Corollary 2.50]{drinfeld2010braided}}\label{SymmCat}
Let $\C$ be a symmetric fusion category. Then one of following holds:

(1)\,  $\C$ is Tannakian

(2)\,  $\C$ is an $\mathbb{Z}_2$-extension of a Tannakian subcategory.
\end{theorem}

The theorem above implies that if $\FPdim(\C)$ is odd then $\C$ is Tannakian. Moreover if $\FPdim(\C)$ is bigger than $2$ then $\C$ necessarily contains a Tannakian subcategory. It is also known that a non-Tannakian symmetric fusion category of Frobenius-Perron dimension $2$ is equivalent to the category $\svect$, see \cite[Subsection 2.12]{drinfeld2010braided}.

\begin{proposition}\label{Adjo_Point}
Let $\C$ be a braided fusion category. Then $\C_{ad}\subseteq (\C_{pt})'$.
\end{proposition}
\begin{proof}
Suppose first that $\C$ is non-degenerate. Then $\C_{ad}= (\C_{pt})'$ by \cite[Corollary 3.27]{drinfeld2010braided}.

Now we suppose that $\C$ is an arbitrary braided fusion category. The braiding of $\C$ induces a canonical embedding of braided fusion categories $\C\hookrightarrow \Z(\C)$. Hence, we may identify $\C$ with a fusion subcategory of $\Z(\C)$. We therefore have $\C_{pt}\subseteq \Z(\C)_{pt}$ and $\C_{ad}\subseteq \Z(\C)_{ad}$. This implies that $(\C_{pt})'\supseteq (\Z(\C)_{pt})'=\Z(\C)_{ad}\supseteq\C_{ad}$. The  equality holds true because $\Z(\C)$ is non-degenerate, also by \cite[Corollary 3.27]{drinfeld2010braided}. This completes the proofs.
\end{proof}

Let $\C$ and $\D$ be two fusion categories with $\Irr(\C)=\{X_1,X_2,\cdots,X_n \}$ and $\Irr(\D)=\{Y_1,Y_2,\cdots,Y_m \}$. The Deligne's tensor product $\C\boxtimes \D$ is a fusion category with simple objects $\Irr(\C\boxtimes \D)=\{X_i\boxtimes Y_j\mid 1\leq i\leq n,1\leq j\leq m\}$ and morphisms $\Hom(X_i\boxtimes Y_j,X_s\boxtimes Y_t)=\Hom(X_i, X_s)\otimes \Hom(Y_j, Y_t)$.

Let $X_1\boxtimes Y_1$ and $X_2\boxtimes Y_2$ be two simple objects of $\C\boxtimes \D$. Then we have
$$(X_1\boxtimes Y_1)\otimes (X_2\boxtimes Y_2)=(X_1\otimes X_2)\boxtimes (Y_1\otimes Y_2).$$

The following theorem is due to Drinfeld et al. In the case when $\C$ is modular, it is due to M\"{u}ger \cite[Theorem 4.2]{muger2003structure}.

\begin{theorem}{\cite[Theorem 3.13]{drinfeld2010braided}}\label{MugerThm}
Let $\C$ be a braided fusion category and $\D$ be a non-degenerate subcategory of $\C$. Then $\C$ is braided equivalent to $\D\boxtimes \D'$, where $\D'$ is the centralizer of $\D$ in $\C$.
\end{theorem}

In \cite{gelaki2008nilpotent}, the authors proved that if $\C$ is a modular category then there is a canonical isomorphism  between $\mathcal{U}(\C)$ and $G(\C)$. The following propositions (Proposition \ref{non_degenerate}, \ref{slig_degen}) show that, in the non-degenerate or slightly degenerate case, $\mathcal{U}(\C)$ and $G(\C)$ also have tight relationships.

\begin{proposition}\label{non_degenerate}
Let $\C$ be a non-degenerate fusion category. Then $|G(\C)|=|\mathcal{U}(\C)|$.
\end{proposition}
\begin{proof}
By \cite[Theorem 3.14]{drinfeld2010braided} and the non-degeneracy of $\C$, we have

\begin{equation}
\begin{split}
&\FPdim(\C_{ad})\FPdim(\C_{ad}')\\
=&\FPdim(\C)\FPdim(\C_{ad}\cap \mathcal{Z}_2(\C))\\
=&\FPdim(\C).
\end{split}\nonumber
\end{equation}

By \cite[Corollary 3.27]{drinfeld2010braided}, we have $\FPdim(\C_{pt})=\FPdim(\C_{ad}')$. This induces an equation $\FPdim(\C_{ad})\FPdim(\C_{pt})=\FPdim(\C)$.
On the other hand, $\FPdim(\C)=|\mathcal{U}(\C)| \FPdim(\C_{ad})$ (see Subsection \ref{sec2.1}). Hence $|G(\C)|=\FPdim(\C_{pt})=|\mathcal{U}(\C)|$.
\end{proof}

For a pair of  fusion subcategories $\A,\B$ of $\D$, we use $\A\vee \B$ to denote the smallest fusion subcategory of $\C$ containing $\A$ and $\B$. The following result will be frequently used in our proof.
\begin{lemma}\cite[Corollary 3.11]{drinfeld2010braided}\label{double_centralzer}
Let $\C$ be a braided fusion category. If $\D$ is any fusion subcategory of $\C$ then $\D''=\D\vee\mathcal{Z}_2(\C)$.
\end{lemma}

\begin{proposition}\label{slig_degen}
Let $\C$ be a slightly degenerate braided fusion category. Then one of the following holds true.

(1)\, $|G(\C)|=|\mathcal{U}(\C)|$. If this is the case then $\mathcal{Z}_2(\C) \nsubseteq \C_{ad}$.

(2)\, $|G(\C)|=2|\mathcal{U}(\C)|$. If this is the case then $\mathcal{Z}_2(\C_{ad})=\mathcal{Z}_2(\C_{ad}^{'})$ contains the category $\svect$.
\end{proposition}
\begin{proof}
By \cite[Proposition 3.29]{drinfeld2010braided}, $\C_{ad}^{'}=\C_{pt}$. Hence \cite[Theorem 3.14]{drinfeld2010braided} shows that
\begin{equation}\label{eq1}
\begin{split}
&\quad\FPdim(\C_{ad})\FPdim(\C_{pt})\\
&=\FPdim(\C_{ad})\FPdim(\C_{ad}^{'})\\
&=\FPdim(\C)\FPdim(\C_{ad}\cap\mathcal{Z}_2(\C)).
\end{split}\nonumber
\end{equation}

Since $\C_{ad}\cap\mathcal{Z}_2(\C)$ is a fusion subcategory of $\mathcal{Z}_2(\C)= \svect$, we have $\C_{ad}\cap \mathcal{Z}_2(\C)=\vect$ or $\svect$.

\medbreak
If $\C_{ad}\cap \mathcal{Z}_2(\C)=\vect$  then $$\FPdim(\C_{ad})\FPdim(\C_{pt})=\FPdim(\C)=\FPdim(\C_{ad})|\mathcal{U}(\C)|.$$
Hence $|G(\C)|=\FPdim(\C_{pt})=|\mathcal{U}(\C)|$. In this case $\mathcal{Z}_2(\C) \nsubseteq \C_{ad}$ since $\C_{ad}\cap \mathcal{Z}_2(\C)=\vect$.

\medbreak
If $\C_{ad}\cap \mathcal{Z}_2(\C)=\svect$ then $$\FPdim(\C_{ad})\FPdim(\C_{pt})=2\FPdim(\C)=2\FPdim(\C_{ad})|\mathcal{U}(\C)|.$$
Hence $|G(\C)|=\FPdim(\C_{pt})=2|\mathcal{U}(\C)|$. Moreover, in this case we have
$$\mathcal{Z}_2(\C_{ad})=\C_{ad}\cap \C_{ad}^{'}=\C_{ad}\cap \C_{pt}\supseteq \C_{ad}\cap \svect=\svect.$$

By Lemma \ref{double_centralzer}, we have
\begin{equation}\label{eq2}
\begin{split}
\mathcal{Z}_2(\C_{ad}^{'})&=\C_{ad}^{'}\cap \C_{ad}^{''}=\C_{ad}^{'}\cap (\C_{ad}\vee \mathcal{Z}_2(\C))\\
&=(\C_{ad}^{'}\cap \C_{ad})\vee \mathcal{Z}_2(\C)= \mathcal{Z}_2(\C_{ad})\vee \mathcal{Z}_2(\C)\\
&=\mathcal{Z}_2(\C_{ad}).
\end{split}\nonumber
\end{equation}

The third equation follows from \cite[Lemma 5.6]{drinfeld2010braided}, $\C$ being braided  and the fact that $\mathcal{Z}_2(\C)=\svect$ is contained in $\C_{ad}^{'}=\C_{pt}$. The last equation follows from the fact that $\mathcal{Z}_2(\C)=\svect$ is a fusion subcategory of $\mathcal{Z}_2(\C_{ad})$.
\end{proof}

Let $\C$ be a fusion category. Then the group $G(\C)$ acts on $\Irr(\C)$ by left(or right) tensor multiplication. For any object $X\in \Irr(\C)$, we shall use $G[X]$ to denote the stabilizer of $X$ under the action of $G(\C)$.

\begin{corollary}\label{slig_degen_triv_grad}
Let $\C$ be a slightly degenerate braided fusion category with trivial universal grading. Then we have

 (1)\, $\C_{pt}=\svect$.

 (2)\, $G[X]$ is trivial for every $X\in \Irr(\C)$.
\end{corollary}
\begin{proof}
 (1)\, By Proposition \ref{slig_degen} and the assumption that $\mathcal{U}(\C)$ is trivial, we have $\FPdim(\C_{pt})=1$ or $2$. On the other hand, $\mathcal{Z}_2(\C)=\svect$ is a fusion subcategory of $\C_{pt}$. Hence $\FPdim(\C_{pt})=2$ and $\C_{pt}=\svect$.

 (2)\, Let $\delta$ be the unique non-trivial simple object generating $\mathcal{Z}_2(\C)=\C_{pt}$. By \cite[Lemma 5.4]{muger2003structure}, $\delta\otimes X\ncong X$ for every $X\in \Irr(\C)$. Therefore, $G[X]$ is trivial for every $X\in \Irr(\C)$.
\end{proof}

Let $\Irr_{\alpha}(\C)$ be the set of nonisomorphic simple objects of Frobenius-Perron dimension $\alpha$.

\begin{lemma}\label{lemma100}
Let $\C$ be a braided fusion category. Suppose that the M\"{u}ger center $\Z_2(\C)$ contains the category $\svect$ of super vector spaces. Then the rank of $\Irr_{\alpha}(\C)$ is even for every $\alpha$.
\end{lemma}
\begin{proof}
Let $\delta$ be the invertible object generating $\svect$, and let $X$ be an element in $\Irr_{\alpha}(\C)$. Then $\delta\otimes X$ is also an element in $\Irr_{\alpha}(\C)$. By \cite[Lemma 5.4]{muger2000galois}, $\delta\otimes X$ is not isomorphic to $X$. This implies that $\Irr_{\alpha}(\C)$ admits a partition $\{X_1,\cdots,X_n\}\cup \{\delta\otimes X_1,\cdots,\delta\otimes X_n\}$. Hence the rank of $\Irr_{\alpha}(\C)$ is even.
\end{proof}

\begin{proposition}\label{cad_non_deg}
Let $\C$ be a braided fusion category with non-trivial universal grading. Assume that $\C_{ad}$ is non-degenerate. Then

(1)\,  $\C_{ad}^{'}=\C_{pt}$ and $\mathcal{Z}_2(\C)=\mathcal{Z}_2(\C_{ad}^{'})$.

(2)\, $(\C_{ad})_{pt}=\vect$. In particular, $\mathcal{U}(\C_{ad})$ is trivial and hence $\C$ is not nilpotent.
\end{proposition}
\begin{proof}
(1)\, Since $\C_{ad}$ is non-degenerate, we have a decomposition $\C=\C_{ad}\boxtimes \C_{ad}^{'}$ by Theorem \ref{MugerThm}. In particular, $\C_{ad}\cap \C_{ad}^{'}=\vect$.

Let $X$ be a simple object of $\C_{ad}^{'}$. Then $X\otimes X^*$ also lies in $\C_{ad}^{'}$. On the other hand, $X\otimes X^*$ lies in $\C_{ad}$ by definition. Hence $X$ must be invertible since  $\C_{ad}\cap \C_{ad}^{'}=\vect$. In other words, $\C_{ad}^{'}$ is fusion subcategory of $\C_{pt}$.

By Proposition \ref{Adjo_Point} and Lemma \ref{double_centralzer}, we have
$$\C_{pt}\supseteq \C_{ad}^{'}\supseteq(\C_{pt})^{''}= \C_{pt}\vee \mathcal{Z}_2(\C)\supseteq \C_{pt}.$$

Hence $\C_{ad}^{'}=\C_{pt}$ and  $\mathcal{Z}_2(\C)\subseteq \C_{pt}$. So we have, also by Lemma \ref{double_centralzer},
\begin{equation}\label{eq0}
\begin{split}
\mathcal{Z}_2(\C_{ad}^{'})&=\C_{ad}^{'}\cap \C_{ad}^{''}=\C_{ad}^{'}\cap(\C_{ad}\vee \mathcal{Z}_2(\C))\\
&=(\C_{ad}^{'}\cap \C_{ad})\vee \mathcal{Z}_2(\C)=\mathcal{Z}_2(\C).
\end{split}\nonumber
\end{equation}

The third equation follows from \cite[Lemma 5.6]{drinfeld2010braided}, $\C$ being braided and the fact that $\mathcal{Z}_2(\C)\subseteq \C_{ad}^{'}=\C_{pt}$

(2)\, By Part (1), we have
$$(\C_{ad})_{pt}=\C_{ad}\cap \C_{pt}=\C_{ad} \cap\C_{ad}^{'}=\vect.$$

By Proposition \ref{non_degenerate}, $|\mathcal{U}(\C_{ad})|=|G(\C_{ad})|=1$. Hence $\C_{ad}$ is not nilpotent, and so is $\C$.  See \cite{gelaki2008nilpotent} for the notion of a nilpotent fusion category.
\end{proof}

\subsection{Equivariantizations and de-equivariantizations}\label{sec23}
Let $\C$ be a fusion category with an action of a finite group $G$. We then can define a new fusion category $\C^G$ of $G$-equivariant objects in $\C$. An object of this category is a pair $(X,(u_g)_{g\in G})$, where $X$ is an object of $\C$, $u_g : g(X)\to X$ is an isomorphism for all $g\in G$, such that
$$u_{gh}\circ \alpha_{g,h} =u_g \circ g(u_h),$$
where $\alpha_{g,h}: g(h(X))\to gh(X)$ is the natural isomorphism associated to the action. Morphisms
and tensor product of equivariant objects are defined in an obvious way. This new category is called the $G$-equivariantization of $\C$.

In the other direction, let $\C$ be a fusion category and let
$\E = \Rep(G)\subseteq \mathcal{Z}(\C)$ be a Tannakian subcategory that embeds into $\C$ via the forgetful
functor $\mathcal{Z}(\C)\to \C$. Here $\mathcal{Z}(\C)$ denotes the Drinfeld center of $\C$. Let $A={\rm Fun}(G)$ be the algebra of functions on $G$. It is a commutative algebra in $\mathcal{Z}(\C)$. Let $\C_G$ denote the category of left $A$-modules in $\C$. It is a fusion category and called the de-equivariantization of $\C$ by $\E$. See \cite{drinfeld2010braided} for details on equivariantizations and de-equivariantizations.

Equivariantizations and de-equivariantizations are inverse to each other:
$$(\C_G)^G\cong\C\cong(\C^G)_G,$$
and their Frobenius-Perron dimensions have the following relations:
\begin{equation}\label{eq11}
\begin{split}
\FPdim(\C)=|G|\FPdim(\C_G) \,\mbox{\,and}\, \FPdim(\C^G)=|G| \FPdim(\C).
\end{split}
\end{equation}

\subsection{Ising categories}\label{sec2.3}
An Ising category $\I$ is a fusion category which is not pointed and $\FPdim(\I )=4$. Let $\I$ be an Ising category then $\Irr(\I)$ consists of three simple objects: $\1$, $\delta$ and $X$, where $\1$ is the unit object, $\delta$ is an invertible object and $X$ is a non-invertible object. They obey the following fusion rules:
$$\delta\otimes \delta\cong \1,\delta\otimes X\cong X\otimes \delta \cong X,X\otimes X\cong \1\oplus\delta.$$

It is easy to check that $\FPdim(\delta)=1$, $\FPdim(X)=\sqrt{2}$ and $X$ is self-dual. It is known that any Ising category $\I$ is a non-degenerate braided fusion category and the adjoint subcategory $\I_{ad}=\I_{pt}$ is braided equivalent to $\svect$. See \cite[Appendix B]{drinfeld2010braided} for more details on Ising categories.

\subsection{Rank $2$ fusion categories}\label{sec2.4}
Ostrik classified rank $2$ fusion categories in \cite{ostrik2003fusion}. Let $\C$ be a rank $2$ fusion category with $\Irr(\C)=\{\1,X\}$. The possible fusion rules for $\C$ are:

$${\rm (1)}\ X\otimes X\cong \1; \quad  {\rm (2)}\ X\otimes X\cong \1\oplus X.$$

If the first possibility holds true then $\C$ is pointed, and hence equivalent to $\vect_{\mathbb{Z}_2}^{\omega}$ for some $\omega\in H^{3}({\mathbb{Z}_2},k^{*})={\mathbb{Z}_2}$. Hence there are two such categories.

If the second possibility holds true then the fusion rules of $\C$ are called the Yang-Lee fusion rules. If this is the case, we call $\C$ a Yang-Lee category or a Fibonacci category. There are two such categories. They are both non-degenerate braided fusion categories. See for a detailed classification of rank $2$ fusion categories.

\section{Extensions of $\vect_{\mathbb{Z}_2}^{\omega}$}\label{sec3}
A generalized Tambara-Yamagami fusion category is a non-pointed fusion category such that, for all non-invertible simple objects $X,Y$, their tensor product $X\otimes Y$ is a direct sum of invertible objects. Generalized Tambara-Yamagami fusion categories were first studied in \cite{liptrap2010generalized}, and then were further studied in \cite{natale2013faithful}.

Let $1 = d_0 < d_1< \cdots < d_s$ be positive real numbers, and let $n_0,n_1,\cdots,n_s$ be positive integers. A fusion category is said of type $(d_0,n_0; d_1,n_1;\cdots;d_s,n_s)$ if $n_i$ is the number of the non-isomorphic simple objects of Frobenius-Perron dimension $d_i$, for all $0\leq i\leq s$.

In the theorem below, we use $0,1$ to denote the elements in $\mathbb{Z}_2$.

\begin{theorem}\label{General_TY_Cat}
Let $\C$ be a $G$-extension of a pointed fusion category $\vect_{\mathbb{Z}_2}^{\omega}$. Then

(1)\, If $\omega=1$ then $\C$ is pointed.

(2)\, If $\omega=0$ then $\C$ is either pointed or a generalized Tambara-Yamagami fusion category. If $\C$ is a generalized Tambara-Yamagami fusion category, then

\mbox{\hspace{0.6cm}}{(i)} $\C$ is of type $(1,2n;\sqrt2,n)$, where $|G(\C)|=2n$.

\mbox{\hspace{0.6cm}}{(ii)} $\vect_{\mathbb{Z}_2}^{0}=\C_{ad}$, $G=\U(\C)$ and $|\mathcal{U}(\C)|=2n$.
\end{theorem}
\begin{proof}
Let $\C=\oplus_{g\in G}\C_g$ be the corresponding faithful grading such that $\C_e=\vect_{\mathbb{Z}_2}^{\omega}$. Since this grading is faithful, every component $\C_g$ has Frobenius-Perron dimension $2$. Since $\C$ is weakly integral, the Frobenius-Perron dimension of every simple object is a square root of some integer \cite[Proposition 8.27]{etingof2005fusion}. This implies that every component $\C_g$ either contains $2$ non-isomorphic invertible objects, or contains a unique $\sqrt 2$-dimensional simple object.

\medbreak
Suppose that $\C$ is not pointed. It follows from \cite[Theorem 3.10]{gelaki2008nilpotent} that $\C$ has a faithful $\mathbb{Z}_2$-grading $\C=\oplus_{h\in \mathbb{Z}_2}\C^h$, where the trivial component $\C^0=\C_{pt}$ is the largest pointed fusion subcategory of $\C$, $\C^{1}$ contains all $\sqrt2$-dimensional simple objects. Let $X,Y$ be non-invertible simple objects of $\C$. Then $X$ and $Y$ lie in $\C^{1}$ and hence $X\otimes Y\in C^0$, which implies that $X\otimes Y$ is a direct sum of invertible objects. This
proves that $\C$ is a generalized Tambara-Yamagami fusion category.

Consider the grading $\C=\oplus_{g\in G}\C_g$. Then the trivial component $\C_e$ is pointed and there exists a component $\C_g$ containing a unique $\sqrt2$-dimensional simple object. It follows from \cite[Lemma 2.6]{jordan2009classification} that $\omega =0$ is trivial. This proves (1) and (2).

\medbreak
Consider the grading $\C=\oplus_{h\in \mathbb{Z}_2}\C^h$ and assume that the number of non-isomorphic invertible simple objects is $m$. Then $\FPdim(\C^0)=m$. Since $\FPdim(\C^0)=\FPdim(\C^{1})$, we get that the number of non-isomorphic $\sqrt2$-dimensional simple objects is $\frac{m}{2}$. Writing $m=2n$, we get that $\C$ is of type $(1,2n;\sqrt2,n)$. This proves (i).

\medbreak
Consider the grading $\C=\oplus_{g\in G}\C_g$. Since $\C_{ad}\subseteq \C_e=\vect_{\mathbb{Z}_2}^{0}$, we know $\C_{ad}=\vect$ or $\vect_{\mathbb{Z}_2}^{0}$. Obviously, $\C_{ad}$ can not be $\vect$, otherwise $\C$ is pointed. Hence $\C_{ad}=\vect_{\mathbb{Z}_2}^{0}$ and $G=\mathcal{U}(\C)$. It follows that $\FPdim(\C)=4n$ and $|\mathcal{U}(\C)|=2n$. This proves (ii).
\end{proof}

For any fusion category $\C$, $\cd(\C)$ denotes the set of Frobenius-Perron dimensions of simple objects of $\C$.

\begin{corollary}\label{necessary_sufficent}
Let $\C$ be a non-pointed fusion category. Then $\C$ is an extension of a rank $2$ pointed fusion category if and only if $\cd(\C)=\{1,\sqrt{2}\}$.
\end{corollary}
\begin{proof}
By Theorem \ref{General_TY_Cat}, it suffices to prove that $\cd(\C)=\{1,\sqrt{2}\}$ implies that $\C$ is an extension of a rank $2$ pointed fusion category.

First, a similar argument as in the proof of Theorem \ref{General_TY_Cat} can prove that $\C$ is a generalized Tambara-Yamagami fusion category. Then, by \cite[Proposition 5.2]{natale2013faithful}, the adjoint subcategory $\C_{ad}$ coincides with the fusion subcategory generated by $G[X]$, for any $\sqrt{2}$-dimensional simple object $X$. Hence $\FPdim(\C_{ad})=2$ and $\C$ is an extension of a rank $2$ pointed fusion category.
\end{proof}

\begin{corollary}\label{transitive}
Let $\C$ be a $G$-extension of $\vect_{\mathbb{Z}_2}^{0}$. Assume that $\C$ is not pointed. Then

(1)\,  The action  of the group $G(\C)$ by left(or right) tensor multiplication on the set $\Irr(\C)-G(\C)$ is transitive.

(2)\, The group $\mathbb{Z}_2$ is a normal subgroup of $G(\C)$.
\end{corollary}

\begin{proof}
Since $\C$ is not pointed, $\C$ is a generalized Tambara-Yamagami fusion category by Theorem \ref{General_TY_Cat}. The corollary then follows from \cite[Lemma 5.1]{natale2013faithful}.
\end{proof}

\begin{proposition}\label{Exist_Ising}
Let $\C=\oplus_{g\in G}\C_g$ be an extension of $\vect_{\mathbb{Z}_2}^{0}$. Then the following are equivalent:

(1)\, $\C$ contains an Ising category.

(2)\, There exists $g\in G$ such that $|g|=2$ and $\rk(\C_g)=1$.

(3)\, There exists a self-dual non-invertible simple object.
\end{proposition}

\begin{proof}
{\rm (1)}$\Rightarrow${\rm (2)}. Let $\I$ be an Ising subcategory of $\C$ and let $X$ be the unique non-invertible simple object of $\I$. Then there exists $g\in G$ such that $X\in \C_g$. Since $\FPdim(X)=\sqrt2$, the proof of Theorem \ref{General_TY_Cat} shows that $\rk(\C_g)=1$. Since $X$ is self-dual, we have $X^{*}\cong X$. On the other hand, $X^{*}\in \C_{g^{-1}}$ and $X\in \C_g$. Hence $\C_g=\C_{g^{-1}}$ which implies that $g=g^{-1}$. This proves that $|g|=2$.

{\rm (2)}$\Rightarrow${\rm (3)}. Let $X$ be the unique simple object of $\C_g$. Then $X^{*}\in \C_{g^{-1}}=\C_g$ because the order of $g$ is $2$. Since the rank of $\C_g$ is $1$, $X^{*}$ must be isomorphic to $X$. This proves that $X$ is self-dual.

{\rm (3)}$\Rightarrow${\rm (1)}. Let $X$ be a self-dual non-invertible simple object of $\C$. Then $X\otimes X^{*}\cong X\otimes X\cong \1\oplus \delta$, where we write $\Irr(\C_e)=\Irr(\C_{ad})=\{\1,\delta\}$. Hence $X$ generates a fusion subcategory $\C\langle X\rangle$ of $\C$ which is not pointed and has Frobenoius-Perron dimension $4$. Hence $\C\langle X\rangle$ is an Ising category, see Subsection \ref{sec2.3}.
\end{proof}

\begin{corollary}\label{Exist_Ising_2}
Let $\C=\oplus_{g\in G}\C_g$ be an extension of $\vect_{\mathbb{Z}_2}^{0}$. Using the same notation as in Theorem \ref{General_TY_Cat}, we assume that $\C$ is of type $(1,2n;\sqrt2,n)$. Then

(1)\, If $n$ is odd then $\C$ contains an Ising category.

(2)\, If $G$ is an elementary abelian $2$-group then $\C$  contains an Ising category.
\end{corollary}

\begin{proof}
The assumption that $n$ is odd implies that $\C$ contains a self-dual non-invertible simple object. The assumption that $G$ is an elementary abelian $2$-group implies that every element $g\in G$ is of order $2$. Therefore, in both cases, the result follows from Proposition \ref{Exist_Ising}.
\end{proof}

\begin{example}
Let $\C$ be a Deligne's tensor product of an Ising category $\I$ and a pointed fusion category $\B$. That is, $\C=\I\boxtimes \B$. Then the Frobenius-Perron dimension of every simple object of $\C$ is $1$ or $\sqrt2$. By Corollary \ref{necessary_sufficent}, $\C$ is an extension of $\vect_{\mathbb{Z}_2}^{0}$.

By \cite[Theorem 5.4]{natale2013faithful}, every non-degenerate generalized Tambara-Yamagami fusion subcategory $\C$ admits a decomposition $\I\boxtimes \B$, hence it is an extension of $\vect_{\mathbb{Z}_2}^{0}$.
\end{example}

\section{Extensions of a Yang-Lee category}\label{sec4}
\begin{lemma}\label{equations}
Let $a,b,c,d,e\geq 3$ be unknowns. Then

(1)\, Equation $$\cos^2\frac{\pi}{a}+\cos^2\frac{\pi}{b}=\frac{5+\sqrt{5}}{8}$$ has unique integral solutions $a=3,b=5$.

(2)\, Equation $$\cos^2\frac{\pi}{c}+\cos^2\frac{\pi}{d}+\cos^2\frac{\pi}{e}=\frac{5+\sqrt{5}}{8}$$ has no integral solutions.
\end{lemma}
\begin{proof}
The function $f(x)=\cos^2\frac{\pi}{x}$ $(x\geq 3)$ is an increasing function on $x$, and $f(10)=\frac{5+\sqrt{5}}{8}$. An easy examination from $x=3$ to $x=10$ proves the lemma.
\end{proof}

\begin{theorem}\label{Exten_YL_Cat}
Let $\C=\oplus_{g\in G}\C_g$ be an extension of a Yang-Lee category $\Y$. Then

(1)\, $\C$ is of type $(1,n;\frac{1+\sqrt5}{2},n)$, where $n=|G(\C)|$.

(2)\, For every $g\in G$, $\rk(\C_g)=2$. Write $\Irr(\C_g)=\{\delta_g,Y_g\}$. Then $\FPdim(\delta_g)=1$ and $\FPdim(Y_g)=\frac{1+\sqrt5}{2}$.

(3)\, $\Y=\C_{ad}$, $G=\U(\C)$ and the order of $\mathcal{U}(\C)$ is $n$.
\end{theorem}
\begin{proof}
Since  $\Y=\C_e$ and the grading is faithful, we have $\FPdim(\C_g)=\FPdim(\Y)=\frac{5+\sqrt5}{2}$, for all $g\in G$. This implies that $\rk(\C_g)\leq 3$ for all $g\in G$. Set $\Irr(\Y)=\{1,Y\}$ with $\FPdim(Y)=\frac{1+\sqrt5}{2}$ (see Subsection \ref{sec2.4}).

\medbreak
If there exists $e\neq g\in G$ such that $\rk(\C_g)=1$ then we set $\Irr(\C_g)=\{Y_g\}$. Then $Y\otimes Y_g \in \C_e\otimes \C_g\subseteq \C_g$. Since $Y_g$ is the unique (non-isomorphic) simple object in $\C_g$, we get that $Y\otimes Y_g\cong \FPdim(Y)Y_g$, which implies that $\FPdim(Y)$ is integral, a contradiction.

\medbreak
If there exists $e\neq g\in G$ such that $\rk(\C_g)=2$ then we set $\Irr(\C_g)=\{\delta_g,Y_g\}$. We may reorder $\delta_g,Y_g$ such that $\FPdim(\delta_g)\leq\FPdim(Y_g)$. Since $\FPdim(\Y)= \FPdim(\C_g)$, we have $\FPdim(\delta_g)^2+\FPdim(Y_g)^2=\frac{5+\sqrt5}{2}$, which implies that $\FPdim(\delta_g)<2$ and $\FPdim(Y_g)<2$. By \cite[Remark 8.4]{etingof2005fusion}, there exist integers $a,b\geq 3$ such that $\FPdim(\delta_g)=2\cos\frac{\pi}{a}$ and $\FPdim(Y_g)=2\cos\frac{\pi}{b}$. Hence we have equation  $4\cos^2\frac{\pi}{a}+4\cos^2\frac{\pi}{b}=\frac{5+\sqrt{5}}{2}$. Lemma \ref{equations} shows that $a=3$ and $b=5$. So $\FPdim(\delta_g)=1$ and $\FPdim(Y_g)=\frac{1+\sqrt{5}}{2}$.

\medbreak
If there exists $e\neq g\in G$ such that $\rk(\C_g)=3$ then we set $\Irr(\C_g)=\{X_g,Y_g,Z_g\}$. Similarly, we have an equation ($c,d,e\geq 3$ are integers):

$$4\cos^2\frac{\pi}{c}+4\cos^2\frac{\pi}{d}+4\cos^2\frac{\pi}{e}=\frac{5+\sqrt{5}}{2}.$$

Lemma \ref{equations} shows that this is impossible. Therefore, every component $\C_g$ is of rank $2$ and $\C$ is of type $(1,n;\frac{1+\sqrt5}{2},n)$, where $n=|G(\C)|$. This proves part (1) and part (2).

Since $\C_{ad}$ is a fusion subcategory of $\Y$ (see Subsection \ref{sec2.1}) and $\Y$ does not have proper fusion subcategory, we have $\C_{ad}=\vect$ or $\Y$. It is clear that $\C_{ad}$ can not be the trivial fusion category $\vect$, otherwise $\C$ is pointed. Hence $\C_{ad}=\Y$ and $G=\mathcal{U}(\C)$ has order $n$.
\end{proof}

In the rest of this section, we will keep notation as in the proof of Theorem \ref{Exten_YL_Cat}.

\begin{corollary}\label{fusionrules}
Let $\C=\oplus_{g\in G}\C_g$ be an extension of a Yang-Lee category $\Y$. Then the fusion rules of $\C$ are:
$$Y_g\otimes Y_h\cong \delta_{gh}\oplus Y_{gh},$$
$$\delta_g\otimes Y_h\cong Y_{gh},Y_h\otimes \delta_g\cong Y_{hg},\delta_g\otimes \delta_h\cong \delta_{gh}.$$
\end{corollary}
\begin{proof}
Since $Y_g\otimes Y_h$ is contained in $\C_{gh}$ and $\C_{gh}$ only contains two non-isomorphic simple objects, a dimension counting shows that $Y_g\otimes Y_h\cong \delta_{gh}\oplus Y_{gh}$. To prove the remained isomorphisms, it suffices to notice that they are simple objects and contained in $\C_{gh}$, $\C_{hg}$ and $\C_{gh}$, respectively.
\end{proof}

\begin{remark}
(1)\, The corollary above shows that the action  of the group $G(\C)$ by left(or right) tensor multiplication on the set $\Irr(\C)-G(\C)$ is transitive. More precisely, $Y_h\cong \delta_{hg^{-1}}\otimes Y_g$ for all $g,h\in G$.

(2)\, It follows from Corollary \ref{fusionrules} that the Grothendieck ring $K(\C)$ of $\C$ is commutative if and only if $G$ is commutative.
\end{remark}

It is easy to check that the map $f:\U(\C)\to G(\C)$ given by $f(g)=\delta_g$ is an isomorphism of groups, by Corollary \ref{fusionrules}. Hence we get the following corollary.

\begin{corollary}\label{universalGraGroup}
The universal grading group $\U(\C)$ is isomorphic to the group $G(\C)$.
\end{corollary}

\begin{example}
Let $\C$ be a Deligne's tensor product of a Yang-Lee category $\Y$ and a pointed fusion category $\B$. It is easy to check that $\C$ admits a faithful $G(\B)$-grading and $\C_e$ is tensor equivalent to $\Y$.

When $n=2$, the fusion category $\C$ is self-dual and admits a $\mathbb{Z}_2$-grading. The Grothendieck ring $K_0(\C)$ is the unique rank $4$ graded self-dual fusion ring which is categorifiable. See \cite{Dong2017Non}.
\end{example}

\begin{corollary}\label{non-pointed_fusionsubcategory}
There is a 1-1 correspondence between the non-pointed fusion subcategories of $\C$ and the subgroups of the universal grading group $\mathcal{U}(\C)$. In particular, if $\mathcal{U}(\C)=\mathbb{Z}_p$ then $\C_{ad}=\Y$ is the unique non-pointed fusion subcategories of $\C$, where $p$ is a prime number.
\end{corollary}

\begin{proof}
Let $\D$ be a non-pointed fusion subcategory of $\C$. For every non-invertible simple object $X\in \D$, Lemma \ref{fusionrules} shows that
\begin{equation}
\begin{split}
X\otimes X^*=\1\oplus Y.
\end{split}\nonumber
\end{equation}
Hence $\C_{ad}$, generated by $\1$ and $Y$, is a fusion subcategory of $\D$. This shows that every non-pointed fusion subcategory of $\C$ contains $\C_{ad}$. On the other hand, any pointed fusion subcategory can not contain $\C_{ad}$. Therefore, the corollary follows from \cite[Corollary 2.5]{drinfeld2010braided}.
\end{proof}

\section{Classification of braided extensions of a rank $2$ fusion category}\label{sec5}
\subsection{Braided extensions of a Yang-Lee category}
\begin{theorem}\label{classification1}
Let $\C=\oplus_{g\in G}\C_g$ be an extension of a Yang-Lee category. Assume that $\C$ is braided. Then $\C\cong \C_{ad}\boxtimes \C_{pt}$, where $\C_{ad}$ is a Yang-Lee category. In this case, $\Z_2(\C)=\Z_2(\C_{pt})$.
\end{theorem}
\begin{proof}
Assume that $\Y:=\C_e$ is a Yang-Lee category. By \cite[Corollary 2.3]{ostrik2003fusion}, $\Y$ is non-degenerate, and hence $\C\cong \Y\boxtimes \B$ by Theorem \ref{MugerThm}, where $\B$ is the M\"{u}ger centralizer of $\Y$ in $\C$. Considering the Frobenius-Perron dimensions of simple objects of $\C$, we can conclude that $\B$ is pointed. Since $\FPdim(\B)=\FPdim(\C_{pt})$, we can get that $\B=\C_{pt}$. Finally, Theorem \ref{Exten_YL_Cat} shows that $\C_{ad}=\Y$. The result that $\Z_2(\C)=\Z_2(\C_{pt})$ is obtained from Proposition \ref{cad_non_deg}.
\end{proof}

\subsection{Braided extensions of $\vect_{\mathbb{Z}_2}^{0}$}

Throughout this section $\C$ will be an extension of  $\vect_{\mathbb{Z}_2}^{0}$ unless explicitly stated. In addition, we assume that $\C$ is braided and not pointed. An object of a fusion category is called trivial if it is isomorphic to $\1^{\oplus n}$ for some natural number $n$.

\medbreak
A simple object $X$ of a fusion category $\D$ is called \emph{faithful} if it generates $\D$ as a fusion category.
\medbreak

\begin{lemma}\label{faithful}
The universal grading group $\U(\C)$ is cyclic if and only if $\C$ has a faithful simple object.
\end{lemma}

\begin{proof}
Since $X\otimes X^*$ is contained in $\C_{pt}$, $\C$ is a nilpotent fusion category of nilpotency class $2$. This lemma then follows from \cite[Theorem 4.1 and Theorem 4.7]{natale2013faithful}.
\end{proof}

A braided fusion category is called \emph{prime} if contains no nontrivial proper non-degenerate fusion subcategory.
\begin{proposition}
Assume that $\U(\C)$ is cyclic. Then $\C$ is prime.
\end{proposition}

\begin{proof}
By Theorem \ref{General_TY_Cat}, we may assume that $\C$ is of type $(1,2n;\sqrt2,n )$. By Lemma \ref{faithful}, $\C$ has a faithful simple object $X$. Then there exists $a\in \U(\C)$ such that $X\in \C_g$. By \cite[Theorem 4.1]{natale2013faithful}, $a$ is the generator of $\U(\C)$. An easy calculation (mainly using $X^{\otimes t}\in \C_{a^t}$) shows that the rank of $\C_{a^{2m-1}}$ is $1$ and the rank of $\C_{a^{2m}}$ is $2$. Moreover, it further shows that every non-invertible simple object is faithful. Notice that the calculation adopts the results from Theorem \ref{General_TY_Cat}: $\C$ is a generated Tambara-Yamagami fusion subcategory, and every component $\C_g$ either contains $2$ non-isomorphic invertible objects, or contains a unique $\sqrt 2$-dimensional simple object. Therefore, the calculation implies that $\C$ contains no proper non-pointed fusion subcategory.

Suppose on the contrary that $\C$ is not prime. Then $\C$ contains a nontrivial proper non-degenerate fusion subcategory $\D$. By the discussion above, $\D$ is pointed. By Theorem \ref{MugerThm}, $\C\cong \D\boxtimes \D'$. Since $\D$ is pointed and not trivial, $\D'$ must be a proper non-pointed fusion subcategory of $\C$. This is impossible by the discussion above.
\end{proof}

\begin{lemma}\label{cad}
The adjoint subcategory $\C_{ad}$ is braided tensor equivalent to $\svect$.
\end{lemma}
\begin{proof}
By Theorem \ref{General_TY_Cat}, we know $\C_{ad}\cong \vect_{\mathbb{Z}_2}^{0}$ is pointed. By \cite[Lemma 2.4]{2016GonNatale}, we have $\C_{ad}$ is symmetric.  Suppose on the contrary that $\C_{ad}$ is Tannakian. Then $\C_{ad}\cong \Rep(\mathbb{Z}_2)$ as braided fusion categories. By \cite[Proposition 2.10]{etingof2011weakly} and \cite[Theorem 4.18(i)]{drinfeld2010braided}, $\C$ is a $\mathbb{Z}_2$-equivariantization of a fusion category $\C_{\mathbb{Z}_2}$ which is called  the de-equivariantization of $\C$ by $\Rep(\mathbb{Z}_2)$.

The forgetful functor $F:\C\rightarrow \C_{\mathbb{Z}_2}$ is a tensor functor and the image of every object in $\C_{ad}$ under $F$ is a trivial object of $\C_{\mathbb{Z}_2}$. Let $\delta$ be the unique non-trivial simple object of $\C_{ad}$. If $X$ is a non-invertible simple object of $\C$ then $X\otimes X^{*}\cong \1\oplus\delta$. Hence $F(X\otimes X^{*})\cong F(X)\otimes F(X)^{*}\cong \1\oplus \1$, which implies that $F(X)$ is not simple. On the other hand, if  $F(X)$ is not simple then the decomposition of $F(X)\otimes F(X)^{*}$ should contain at least four simple summands. This contradiction shows that the assumption is false, and hence $\C_{ad}\cong \svect$.
\end{proof}

\begin{lemma}\label{cpt}
Assume that the largest pointed fusion subcategory $\C_{pt}$ of $\C$ is slightly degenerate. Then $\C$ contains an Ising category. In particular, $\C$ is non-degenerate.
\end{lemma}
\begin{proof}
By \cite[Proposition 2.6(ii)]{etingof2011weakly}, $\C_{pt}\cong \svect\boxtimes \mathcal{B}$ as braided tensor categories, where $\mathcal{B}$ is a pointed non-degenerate fusion category. In particular $\FPdim(\mathcal{B})=n$ since $\C$ is of type $(1,2n;\sqrt2,n )$ (see Theorem \ref{General_TY_Cat}). Let $\mathcal{B}^{'}$ be the centralizer of $\mathcal{B}$ in $\C$. By Theorem \ref{MugerThm}, $\C\cong \mathcal{B}\boxtimes \mathcal{B}^{'}$ as braided fusion categories. In particular, $\C$ is non-degenerate if and only if $B^{'}$ is. Counting dimensions on both sides, we get that $\FPdim(\mathcal{B}^{'})=4$. Since $\mathcal{B}^{'}$ is not pointed, $\mathcal{B}^{'}$ is an Ising category. By \cite[Corollary B.12]{drinfeld2010braided}, $\mathcal{B}^{'}$ is non-degenerate and hence $\C$ is also non-degenerate.
\end{proof}

Let $\D$ be a fusion category with commutative Grothendieck ring, and let $\A$ be a fusion subcategory of $\D$. The \emph{commutator} of $\A$ in $\D$ is denoted by $\A^{co}$; that is, $\A^{co}$ is the fusion subcategory of $\D$ generated by all simple objects $X$ of $\D$ such that $X\otimes X^*$ is contained in $\A$ \cite{gelaki2008nilpotent}.

\begin{lemma}\label{cad_cpt}
The adjoint subcategory $\C_{ad}$, the largest pointed fusion subcategory $\C_{pt}$ and the M\"{u}ger center $\mathcal{Z}_2(\C)$ have the following relations.

(1)\, $\C_{ad}^{'}=\C_{pt}$ and $\mathcal{Z}_2(\C)\subseteq \C_{pt}$.

(2)\, $\mathcal{Z}_2(\C_{pt})=\C_{ad}\vee \mathcal{Z}_2(\C).$
\end{lemma}
\begin{proof}
(1)\,  By \cite[Proposition 3.25]{drinfeld2010braided}, a simple object $X\in \C$ belongs to $\C_{ad}'$ if and only if it belongs to $\mathcal{Z}_2(\C)^{co}$; that is, if and only if $X\otimes X^{*}\in \mathcal{Z}_2(\C)$. If $X$ is not invertible then $X\otimes X^{*}\cong \1\oplus \delta$ and hence $\delta\otimes X\cong X$, where $\delta$ is unique non-trivial simple object of $\C_{ad}\cong \svect$ by Lemma \ref{cad}. This implies that $\svect\subseteq \mathcal{Z}_2(\C)$. This is impossible by \cite[Lemma 5.4]{muger2000galois} which says that if $\svect\subseteq \mathcal{Z}_2(\C)$ then $\delta \otimes Y\ncong Y$ for any $Y\in \C$. Therefore, $\C_{ad}^{'}\subseteq \C_{pt}$ is pointed. By Proposition \ref{Adjo_Point}, $\C_{ad}^{'}\supseteq \C_{pt}^{''}=\C_{pt}\vee \mathcal{Z}_2(\C)$. Hence we have
$$\C_{pt}\supseteq \C_{ad}^{'}\supseteq\C_{pt}\vee \mathcal{Z}_2(\C)\supseteq \C_{pt}.$$

This shows that $\C_{ad}^{'}=\C_{pt}$ and $\mathcal{Z}_2(\C)\subseteq \C_{pt}$.

(2) By Part (1), we have
\begin{equation}
\begin{split}
&\mathcal{Z}_2(\C_{pt})=\C_{pt}\cap \C_{pt}^{'}=\C_{pt}\cap \C_{ad}^{''}\\
&=\C_{pt}\cap (\C_{ad}\vee \mathcal{Z}_2(\C))=\C_{ad}\vee \mathcal{Z}_2(\C).
\end{split}\nonumber
\end{equation}

The last equality follows from \cite[Lemma 5.6]{drinfeld2010braided} and the fact that $\C$ is braided. Together with Lemma \ref{cad}, we get $\mathcal{Z}_2(\C_{pt})=\svect\vee \mathcal{Z}_2(\C)$.
\end{proof}

\begin{lemma}\label{Exist_Isingnew}
Suppose that $\FPdim(\C)=8$. Then the fusion category $\C$ contains an Ising category.
\end{lemma}
\begin{proof}
In this case,  $\C$ is of type $(1,4;\sqrt2,2 )$. Since $\mathcal{Z}_2(\C_{pt})$ is a fusion subcategory of $\C_{pt}$, $\FPdim(\mathcal{Z}_2(\C_{pt}))$ divides $\FPdim(\C_{pt})$. By Lemma \ref{cad_cpt}, $\mathcal{Z}_2(\C_{pt})$ contains $\svect$. Hence $\FPdim(\mathcal{Z}_2(\C_{pt}))=2$ or $4$.

If $\FPdim(\mathcal{Z}_2(\C_{pt}))=2$ then $\mathcal{Z}_2(\C_{pt})=\svect$ and hence $\C_{pt}$ is slightly degenerate. So $\C$ contains an Ising category by Lemma \ref{cpt}. If $\FPdim(\mathcal{Z}_2(\C_{pt}))=4$  then $\C_{pt}$ contains a Tannakian subcategory $\D=\Rep(\mathcal{Z}_2)$ by Theorem \ref{SymmCat}. Then the de-equivariantization of $\C$ by $\D$ is non-pointed by \cite[Corollary 2.2]{burciu2013fusion}, and of dimension $4$ (see equation (\ref{eq11})). Hence it is an Ising category, see Section \ref{sec2.3}. By \cite[Lemma 4.7]{2016BruPlaRow}, $\C$ also contains an Ising category.
\end{proof}

\begin{lemma}\label{Exist_Ising31}
Suppose that $\E\subseteq \mathcal{Z}_2(\C)$ is a non-trivial Tannakian subcategory. Then the fusion category $\C$ contains an Ising category.
\end{lemma}

\begin{proof}
 We prove the theorem by induction on the dimension of $\C$.

\medbreak
By Theorem \ref{General_TY_Cat}, we may assume that $\C$ is of type $(1,2n;\sqrt2,n )$, and hence $\FPdim(\C)=4n$. When $n=1$ then $\C$ is of type $(1,2;\sqrt2,1 )$, and hence $\C$ is an Ising category. In this case $\mathcal{Z}_2(\C)$ is trivial. So $n\geq 2$.

When $n=2$ the lemma follows from Lemma \ref{Exist_Isingnew}. This completes the proof of basic step. In the rest of our proof, we assume that $n\geq 2$.

\medbreak
By assumption, there exists a group $G$ such that $\Rep(G)\cong \E$ as symmetric fusion categories. Let $\A=\C_{G}$ be the de-equivariantization of $\C$ by $\E$. Then $\A$ is braided by \cite[Remark 2.3]{etingof2011weakly}. By \cite[Corollary 2.2]{burciu2013fusion}, the Frobenius-Perron dimension of every simple object of $\A$ is $1$ or $\sqrt2$. It follows from Corollary \ref{necessary_sufficent} that $\A$ is an $H$-extension of a rank $2$ pointed fusion category for some finite group $H$. Since  $\FPdim(\A)= \frac{\FPdim(\C)}{|G|}$, the induction shows that $\A$ contains an Ising cateogry. By \cite[Lemma 4.7]{2016BruPlaRow}, $\C$ also contains an Ising category.
\end{proof}

\begin{lemma}\label{Exist_Ising32}
Suppose that $\C$ is slightly degenerate and $\U(\C)$ is not cyclic. Then the fusion category $\C$ contains an Ising category.
\end{lemma}
\begin{proof}
Let $X$ be a $\sqrt{2}$-dimensional simple object of $\C$, and let $\D=\langle X\rangle$ be the subcategory generated by $X$. Then $\D$ is a non-pointed fusion subcategory of $\C$. By Lemma \ref{faithful}, $X$ is not faithful and hence $\D$ is a proper fusion subcategory of $\C$. We will prove that $\D$ contains an Ising category.

\medbreak
If $\mathcal{Z}_2(\D)=\vect$ then $\D$ is non-degenerate then $\D$ contains an Ising category by \cite[Theorem 5.4]{natale2013faithful}. If $\mathcal{Z}_2(\D)=\svect$ then $\D$ is also slightly degenerate. By induction on the dimension, $\D$ contains an Ising category. In other cases, $\mathcal{Z}_2(\D)$ contains a non-trivial Tannakian subcategory. Lemma \ref{Exist_Ising31} shows in this case that $\D$ contains an Ising category.
\end{proof}

\begin{remark}
Assume that $\C$ is of type $(1,2n;\sqrt2,n )$. By Lemma \ref{lemma100}, if $\C$ is slightly degenerate then $n$ must be even.
\end{remark}

\begin{theorem}\label{classification2}
Let $\C=\oplus_{g\in G}\C_g$ be an extension of $\vect_{\mathbb{Z}_2}^{\omega}$ for some $\omega$. Assume that $\C$ is braided. Then $\C$ is exactly one of the following:

(1)\,  $\C$ is pointed.

(2)\,  $\C$ is slightly degenerate and generated by a $\sqrt{2}$-dimensional simple object. In this case, $\C$ is prime.

(3)\,  $\C\cong \I\boxtimes \B$, where $\I$ is an Ising category, $\B$ is a braided pointed fusion category.
\end{theorem}
\begin{proof}
The Theorem \ref{General_TY_Cat} shows that $\C$ is either pointed, or a generalized Tambara-Yamagami fusion subcategory. This proves Part (1). To prove Parts (2) and (3), it is enough to consider the case when $\C$ is a generalized Tambara-Yamagami fusion subcategory.

We assume that Part (2) does not hold true and prove Part (3). It suffices to prove that $\C$ contains an Ising category. In fact, if $\C$ contains an Ising category $\I$, then $\C\cong \I\boxtimes \B$ by Theorem \ref{MugerThm}, where $\B$ is the M\"{u}ger centralizer of $\I$ in $\C$. Considering the Frobenius-Perron dimensions of simple objects of $\C$, we can conclude that $\B$ is pointed. Our proof is conducted by considering the M\"{u}ger center $\mathcal{Z}_2(\C)$ of $\C$.

If $\mathcal{Z}_2(\C)$ is trivial then $\C$ is non-degenerate then \cite[Theorem 5.4]{natale2013faithful} shows that $\C$ contains an Ising category. If $\mathcal{Z}_2(\C)$ contains a non-trivial Tannakian subcategory then Lemma \ref{Exist_Ising31} shows that $\C$ contains an Ising category. If $\C$ is slightly degenerate and not generated by a $\sqrt{2}$-dimensional simple object then $\C$ contains an Ising category by Lemma \ref{Exist_Ising32}. This finishes the proof.
\end{proof}
\begin{corollary}\label{extension_of_Ising}
Let $\D$ be an extension of an Ising category $\I$. Assume that $\D$ is braided. Then $\D\cong\I\boxtimes \B$, where $\B$ is a braided pointed fusion category.
\end{corollary}
\begin{proof}
If $\I$ is the adjoint subcategory of $\D$ then Proposition \ref{cad_non_deg} shows that $\I_{pt}$ should be trivial. This is a contradiction. Hence $\D_{ad}$ should be a proper fusion subcategory of $\I$. On the other hand, $I_{pt}$ is the only proper fusion subcategory of $\I$. Hence $\D_{ad}=\I_{pt}$ and $\D$ is an extension of a pointed fusion category of rank $2$. Moreover, $\D$ is not pointed since $\I$ is not pointed. Therefore $\D\cong\I\boxtimes \B$ by Theorem \ref{classification2}, where $\B$ is a braided pointed fusion category.
\end{proof}

\section*{Acknowledgements}
The research of the first author is partially supported by the startup foundation for introducing talent of NUIST and the Natural Science Foundation of China (Grant No. 11201231).



\end{document}